\documentclass[sn-mathphys-num]{sn-jnl}

\usepackage{graphicx}%
\usepackage{multirow}%
\usepackage{amsmath,amssymb,amsfonts}%
\usepackage{amsthm}%
\usepackage{mathrsfs}%
\usepackage[title]{appendix}%
\usepackage{xcolor}%
\usepackage{textcomp}%
\usepackage{manyfoot}%
\usepackage{booktabs}%
\usepackage{algorithm}%
\usepackage{algorithmicx}%
\usepackage{algpseudocode}%
\usepackage{listings}%

\theoremstyle{thmstyleone}%
\newtheorem{theorem}{Theorem}[section]
\newtheorem{proposition}[theorem]{Proposition}%
\newtheorem{lemma}[theorem]{Lemma}%
\newtheorem{corollary}[theorem]{Corollary}%

\theoremstyle{thmstyletwo}%
\newtheorem{example}{Example}%
\newtheorem{assumption}{Assumption}%

\theoremstyle{thmstylethree}%

\newcommand\bv{\mathbf{v}}

\newcommand\x{\mathbf{x}}
\newcommand\m{\mathbf{m}}

\newcommand\bp{\mathbf{p}}
\newcommand\A{\mathbf{A}}

\newcommand\R{\mathbb{R}}

\newcommand\N{\mathbb{N}}

\newcommand\M{\mathbf{M}}

\newcommand\B{\mathbf{B}}
\newcommand\q{\mathbf{q}}

\newcommand\balpha{{\boldsymbol{\alpha}}}
\newcommand\bbeta{{\boldsymbol{\beta}}}
\newcommand\bgamma{{\boldsymbol{\gamma}}}

\newcommand\bpsi{{\boldsymbol{\psi}}}

\newcommand\bphi{\boldsymbol{\phi}}
\newcommand\bnu{\boldsymbol{\nu}}
\newcommand\bmu{\boldsymbol{\mu}}

\begin{document}

\title[Convex relaxations for the total variation distance]
{A hierarchy of convex relaxations for the total variation distance}


\author*[1,2]{\fnm{Jean B. }\sur{Lasserre}}\email{lasserre@laas.fr}

%

\affil*[1]{\orgdiv{LAAS-CNRS}, \orgaddress{\street{7 Av. Colonel Roche}, \city{Toulouse}, \postcode{31031},  \country{France}}}

\affil[2]{\orgdiv{Dept. Math}, \orgname{Toulouse School of Economics}, \orgaddress{\city{Toulouse}, \postcode{31000}, \country{France}}}



\abstract{Given two measures $\mu,\nu$ on $\R^d$ that  satisfy Carleman's condition, we provide a numerical 
scheme to approximate as closely as desired the total variation distance between $\mu$ and $\nu$.
(In particular, the supports of $\mu$ and $\nu$ are not necessarily compact.)
It consists of solving a sequence (hierarchy) of convex relaxations whose associated sequence of optimal values
converges to the total variation distance, an additional illustration
of the versatility of the Moment-SOS hierarchy.
Each relaxation in the hierarchy
is a semidefinite program whose size increases 
with the number of involved moments. It has an optimal solution which 
is a couple of degree-$2n$ pseudo-moments which converge, as $n$ grows,  to 
moments of the Hahn-Jordan decomposition of $\mu-\nu$. Illustrative examples are provided.}




\maketitle

\section{Introduction}\label{sec1}

This paper is concerned with the numerical evaluation of the 
\emph{total variation} distance between two given probability measures, based on knowledge of their moments.

Evaluating a ``distance" between measures is an important topic with many applications, 
e.g. for homogeneity testing and independence testing  as advocated in \cite{ers},
for defining ambiguity sets in distributionally robust optimization 
\cite{dist-robust1,dist-robust2,dist-robust3,dist-robust4}, and has also become increasingly important in Data Science and Machine Learning in particular.
Among possible choices, 
the family of \emph{integral probability metrics} (IPM)  which includes the
Kantorovich, Dudley, Kolmogorov and total variation (TV) metrics, is discussed in \cite{ers} 
where the authors provide several empirical estimators of the associated distances between two distributions, based on random i.i.d. samples. See also \cite{entropy} for a discussion on relative merits of several distances.  

In particular, the Kantorovich metric (dual to Wasserstein distance) has become popular and one reason
is that its optimal transport formulation allows to define
efficient specialized procedures (e.g. the Sinkhorn algorithm)
for its computation \cite{cuturi}.  On the other hand, as the TV  distance 
is the same as the Wasserstein distance with (nasty) cost 
function $c(x,y)=1_{x\neq y}(x,y)$, it is an indication that its effective computation is a computational challenge. For instance, in \cite{ers} where the authors provide several empirical estimators of integral probability metrics (IPMs), when specializing to TV distance the resulting estimator
is not consistent, and for this reason the authors provide {\color{blue} lower} bounds \cite[Proposition 5.1]{ers}.
The reason is that the set of bounded measurable functions of norm $1$ is too large for efficient evaluation of
$TV(\mathbb{P},\mathbb{Q})=\sup_f \{\,\left\vert \int f\,d\mathbb{P}-\int f\,d\mathbb{Q}\right\vert:\Vert f\Vert_\infty\leq 1\,\}$ for two distributions $\mathbb{P}$ and $\mathbb{Q}$.
In view of such difficulties, recent contributions have focused on providing analytical upper and/or lower 
bounds on $TV(\mathbb{P},\mathbb{Q})$ for $\mathbb{P},\mathbb{Q}$ in some classes 
of distributions, e.g. two high-dimensional gaussians with same mean in \cite{devroye-2}, or mixture of two Gaussians with same covariance matrix in \cite{PMLR22}, or two arbitrary measures with given means and variance in \cite{Nishiyama}; recently in \cite{gaussian-high} 
the authors provide a tight (up to a constant factor) lower bound on the TV distance for
high-dimensional gaussians.  Finally, let us mention Pinsker's inequality 
$\Vert\mu-\nu\Vert_{TV}\leq \sqrt{D_{KL}(\mu\Vert\nu)/2}$ 
which provides
an upper bound on the TV distance via the Kullback-Leibler divergence \cite[\S 3.1]{csiszar}, and the bounds
$H(\mu,\nu)^2\leq \Vert\mu-\nu\Vert_{TV}\leq \sqrt{2}H(\mu,\nu)$ 
via  the Hellinger distance $H$ \cite[Chapter 2]{Tsybakov}.

In another direction, in \cite{barron}  the authors consider estimators of an unknown distribution $\mu$ and, in view of \cite{devroye-1}, advocate that some \`a priori information on $\mu$ is required if the estimators are required to be consistent in total variation. Then 
 under the assumption that the non-atomic part of $\mu$ is absolutely continuous with respect to some \`a priori known $\sigma$-finite measure, they provide estimators which are consistent in total variation 
 (a.s. and in expectation).\\

{\bf Contribution.} In this paper we show
that the \emph{total variation} distance is amenable to practical computation under relatively weak assumptions and so could provide an alternative to other distances when needed.
In a rather general context, we provide a numerical scheme to approximate as closely as desired the total variation distance between two measures $\mu$ and $\nu$.  
We do \emph{not} assume that $\mu$ or $\nu$ has compact support, but
we assume that all moments of $\mu$ and $\nu$ are finite, and that both $\mu$ and $\nu$ satisfy Carleman's condition.  We formulate the problem as an infinite-dimensional linear program (LP) on a space of measures, with 
an important constraint of domination inherited from the Hahn-Jordan decomposition of $\mu-\nu$. 
This LP-formulation is then viewed as an instance of the 
Generalized Moment Problem (GMP) with polynomial data, so that
the resulting GMP  is amenable to practical computation via the Moment-SOS 
hierarchy \cite{ICM-2018,HKL,lasserre-acta}.  As a result, one may approximate as closely as desired $\Vert\mu-\nu\Vert_{TV}$ as more and more moments of $\mu$ and $\nu$ are taken into account. More precisely:

(i) Our  numerical scheme 
consists of solving a sequence (hierarchy) of convex relaxations.
Each convex relaxation of the hierarchy is a semidefinite program\footnote{A semidefinite program is a convex conic optimization problem that can be solved efficiently, up to arbitrary precision fixed in advance; see e.g. \cite{anjos-lasserre}} whose size increases 
with the number of moments of $\mu$ and $\nu$ involved.  

(ii) The associated sequence of optimal values is monotone non decreasing and converges from below to $\Vert\mu-\nu\Vert_{TV}$. Crucial for convergence is a domination constraint coming from a property of the Hahn-Jordan decomposition of $\mu-\nu$.

(iii) 
The associated sequence of 
optimal solutions of relaxations (a couple of pseudo-moment vectors whose size increases), converges to the unique couple of infinite moment vectors
of the Hahn-Jordan decomposition $(\phi^*_+,\phi^*_-)$ of the signed measure $\mu-\nu$.

(iv) Each semidefinite relaxation of the hierarchy has a dual semidefinite program, very much in the spirit of
the classical TV-distance dual formulation
\begin{equation}
\label{intro-dual-classic}
\Vert\mu-\nu\Vert_{TV}\,=\,\sup_{f}\,\{\,\int f\,\,d\mu -\int f\,\,d\nu:\: \Vert f\Vert_\infty\leq 1\,\}\,
\end{equation}
where the ``sup" is over bounded measurable functions. Our hierarchy of duals shows
how the above classical formulation can be strengthened by (i) restricting to polynomials and (ii),
including an additional penalized integral term 
(w.r.t. $\mu$ and $\nu$) in the criterion. This term penalizes the unavoidable 
violation of the constraint $\Vert f\Vert_\infty\leq1$
when $f$ is a polynomial, and corresponds to the domination constraint in the primal formulation. 

(v) It turns out that when $\mu$ and $\nu$ are 
measures on the real line, our first lower bound with $n=1$ in the hierarchy (i.e. when one uses moments up to degree $2n=2$ only)
coincides with 
the analytical lower bound provided in \cite{Nishiyama} and based solely on the means and variances of $\mu$ and $\nu$.
As shown on some examples, the improvement is already significant with $n=2$ (i.e. by now taking into account moments up to degree $4$) and even better with $n=3,4$. 

Moreover, and as a nice feature of our numerical scheme,
we prove that for two atomic probability measures 
respectively supported on $m_1$ and $m_2$ atoms of the real line, 
the exact distance $\Vert\mu-\nu\Vert_{TV}$ is obtained as soon as the degree $n$ of the semidefinite relaxation in the hierarchy, matches $\max[m_1,m_2]$, i.e.,  when the minimal information required is used.
Hence, for instance, mutual singularity (if any) (i.e., $\Vert\mu-\nu\Vert_{TV}=2$) is detected at $n=\max[m_1,m_2]$. 
In addition, in principle
no geometric condition on a separation of the respective atoms of $\mu$ and $\nu$ is required and 
this nice feature is illustrated on a toy example with $\mu$ the Dirac $\delta_0$ at $x=0$ and $\nu$
the Dirac $\delta_\varepsilon$ at $x=\varepsilon$ (with arbitrary small $\varepsilon>0$). (However 
as in practice one uses a numerical semidefinite solver, the issue of requiring a minimum separation of the atoms becomes relevant due to unavoidable potential numerical inaccuracies.)

(vi) We also provide a set of illustrative numerical experiments 
to illustrate (a) our result on discrete measures on the real line,  and  (b) the behavior
of the algorithm  when $\mu$ and $\nu$ are two univariate Gaussian $\mathcal{N}(m_1,\sigma_1)$
and $\mathcal{N}(m_2,\sigma_2)$. 

(vii) Finally, it is worth emphasizing that the optimal value of each relaxation provides a 
\emph{guaranteed} lower bound on the TV distance
which increases with the degree of the relaxation.
This information already provided at early steps of the hierarchy should be useful
because in view of the current status of semidefinite solver software packages, one cannot expect to solve high degree relaxations, even for relatively modest dimensions.

At last but not least, the input data required at the $n$-th semidefinite relaxation of the hierarchy
is the \emph{finite} set of degree-$2n$ moments of $\mu$ and $\nu$,
 assumed to be known\footnote{For instance if $\mu$ and $\nu$ are two Gaussians 
 $\mathcal{N}(\m,\Sigma)$ and $\mathcal{N}(\m',\Sigma')$ respectively, their 
 moments are known explicitly in terms of  $\m,\m'$ and the entries of $\Sigma$ and $\Sigma'$.
 The same is true e.g. for pairs of exponential measures, or gaussian mixtures. }
  or estimated from random  i.i.d. 
samples drawn from $\mu$ and $\nu$. In the latter case, by the SLLN, such a finite set of degree-$2n$ moments can be estimated  as closely as desired and almost surely, provided that the sample size is sufficiently large.
Then the true moment matrices $\M_n(\mu)$ and $\M_n(\nu)$ of $\mu$ and $\nu$
needed in the $n$-th semidefinite relaxation of our numerical scheme, can be safely replaced with their analogues $\M_n(\mu^N)$ and $\M_n(\nu^N)$ obtained from the empirical measures $\mu^N$ and $\nu^N$ associated with a  sample  of size $N$. Of course, when $n$ increases, the sample size $N$ needs to be adjusted with the number of degree-$2n$ moments considered. This issue was also analyzed in \cite{CUP2022} to analyze the respective behavior 
of the Christoffel functions respectively associated with a measure $\mu$ and its empirical version $\mu^N$ from a sample.

Hence in summary, our contribution is to provide
an additional tool in the arsenal of algorithms available in applied probability, for approximating as closely as desired,
the total variation distance $\Vert\mu-\nu\Vert_{TV}$ based on moment information. This tool can thus be applied 

--  not only in applications where moments of $\mu$ and $\nu$ are available in closed form
(e.g. for $\mu$ and $\nu$ Gaussian or exponentials  (and their mixtures)), but also

-- even in applications where only random i.i.d. samples from $\mu$ and $\nu$ are available.
Indeed as already mentioned, with fixed $n$, the finite set of $2n$-degree empirical moments obtained from a sample, can approximate as closely  as desired the same set of true degree-$2n$ moments, provided the sample size is sufficiently large (hence adapted to the degree $n$ considered). 

As a technical comment, we wish to also emphasize the relatively weak assumption on the measures $\mu,\nu$, namely that they satisfy Carleman's condition
(no compact support is required).  Crucial in our numerical scheme are the two
domination constraints  $\phi^+\leq\mu$ and $\phi^-\leq\nu$  where $(\phi^+,\phi^-)$ is the Hahn-Jordan
decomposition of the signed measure $\mu-\nu$. While redundant  in the infinite-dimensional 
GMP formulation, they become extremely useful (as a compactification tool) in the relaxation scheme.
Interestingly, the effect of such domination constraints is also revealed in the dual problem
at step $n$ of the hierarchy when this dual is compared with the classical dual formulation \eqref{intro-dual-classic} of the TV distance.

In a final remark, as an alternative to algorithms based on discretizations (like e.g. Sinkhorn algorithm),
the Wasserstein distances $W_2(\mu,\nu)$ and $W_{1}(\mu,\nu)$ (with \emph{polynomial} or piecewise-polynomial cost $c(x,y)$) can also be approximated as closely as desired in  a mesh-free practical computation by (i) applying the Moment-SOS hierarchy \cite{HKL,CUP} for solving the associated optimal transport problem (OT), and (ii) extract the transport map from the moment vector solution of the OT, by a non-standard application of the Christoffel-Darboux kernel \cite{HL-constructive}; see e.g. the recent work of \cite{nouy-mula} for such an approach. However, crucial in \cite{HKL,CUP,nouy-mula} is the fact that the cost function is a polynomial 
(or piecewise-polynomial) and the supports are compact, which of course excludes the nasty cost function $1_{x\neq y}(x,y)$ in the TV distance formulation, let alone the non-compact supports of the involved measures.

\section{Main result}
\subsection{Notation and definitions}
Let $\R[\x]$ denote the ring of real polynomials in the variables $(x_1,\ldots,x_d)$ and $\R[\x]_n\subset\R[\x]$ be its subset 
of polynomials of total degree at most $n$. 
Let $\N^d_n:=\{\balpha\in\N^d: \sum_i\alpha_i\leq n\}$
with cardinal $s(n)={n+d\choose n}$. Let $\bv_n(\x)=(\x^{\balpha})_{\balpha\in\N^d_n}$ 
be the vector of monomials up to degree $n$, 
and let $\Sigma[\x]_n\subset\R[\x]_{2n}$ be the convex cone of polynomials of total degree at most $2n$ which are sum-of-squares (in short SOS).  A polynomial $p\in\R[\x]_n$ can be identified with its vector of
coefficients $\bp=(p_{\balpha})\in\R^{s(n)}$ in the monomial basis, and reads
\[\x\quad\mapsto p(\x)\,:=\,\langle\bp,\bv_n(\x)\rangle\,,\quad \forall p\in\R[\x]\,.\]
Denote by $\mathscr{M}(\R^d)$ (resp. $\mathscr{M}(\R^d)_+$) the space of signed (resp. positive) Borel measures on $\R^d$. For two Borel measures $\mu,\nu\in\mathscr{M}(\R^d)_+$, the notation $\mu\leq\nu$ stands for
$\mu(B)\leq\nu(B)$ for all Borel sets $B\in\mathcal{B}(\R^d)$. The support of a Borel measure $\mu$ on $\R^d$ is the smallest closed set $A$ such that
$\mu(\R^d\setminus A)=0$, and such a  set $A$ is unique. A Borel measure whose all moments are finite is said to be (moment) \emph{determinate} if there is no other measure with same moments.  

For a real symmetric matrix 
$\A=\A^T$, the notation $\A\succeq0$ (resp. $\A\succ0$) stands for $\A$ is positive semidefinite (p.s.d.) (resp. positive definite (p.d.)).

\noindent
{\bf Hahn-Jordan decomposition.}  Given two finite Borel measures $\mu,\nu\in\mathscr{M}(\R^d)_+$, the signed measure $\mu-\nu$ has a unique Hahn-Jordan decomposition $(\phi^*_+,\phi^*_-)$ such that $\phi^*_+-\phi^*_-=\mu-\nu$. 
That is, there exists a Borel set $A\in\mathcal{B}(\R^d)$ and two mutually singular\footnote{Two positive measures $\mu$ and $\nu$ on $\R^d$ are mutually singular (noted $\mu\perp\nu$), if there exist two disjoint Borel sets $F,G\subset\R^d$ such that $\R^d=F\cup G$,
$\mu(G)=0$ and $\nu(F)=0$.} positive measures $\phi^*_+,\phi^*_-$ 
such that $\phi^*_+(\R^d)=\phi^*_+(A)$ while $\phi^*_-(A)=0$, 
and
\begin{equation}
\label{Hahn}
\phi^*_+(B)\,=\,(\mu-\nu)(B\cap A)\,;\quad\phi^*_-(B)\,=\,(\nu-\mu)(B\cap (\R^d\setminus A))\,,\quad\forall B\in\mathcal{B}(\R^d)\,.\end{equation}
In addition, and obviously, $\Vert\mu-\nu\Vert_{TV}\leq \mu(1)+\nu(1)$. Moreover,  observe that $\phi^*_+\leq\mu$ and $\phi^*_-\leq\nu$. This property will  turn  out to be crucial for convergence of our numerical scheme.\\

\noindent
{\bf Riesz linear functional and moment matrix.} 
With a real sequence $\bphi=(\phi_{\balpha})_{\balpha\in\N^d}$  
(in bold) is associated the \emph{Riesz} linear functional $\phi\in \R[\x]^*$  (not in bold) defined by
\[p\:(=\sum_{\balpha}p_{\balpha}\x^{\balpha})\quad \mapsto \phi(p)\,=\,\langle\bphi,\bp\rangle\,=\,\sum_{\balpha}p_{\balpha}\,\phi_{\balpha}\,,\quad\forall p\in\R[\x]\,,\]
and the moment matrix $\M_n(\bphi)$
with rows and columns indexed by $\N^d_n$ (hence of size $s(n)$), and with entries
$\M_n(\bphi)(\balpha,\bbeta):=\phi(\x^{\balpha+\bbeta})=\phi_{\balpha+\bbeta}$, $\balpha,\bbeta\in\N^d_n$.
Notice that one may write indifferently $\M_n(\bphi)$ 
or $\M_n(\phi)$, i.e., referring to the sequence $\bphi$
truncated to degree-$2n$ moments or to the Riesz linear functional $\phi$ associated with $\bphi$.

A  real sequence $\bphi=(\phi_{\balpha})_{\balpha\in\N^d}$ has a representing measure 
if its associated linear functional $\phi$ is a Borel measure on $\R^d$. In this case 
$\M_n(\bphi)\succeq0$ for all $n$; the converse is not true in  general.

\noindent
{\bf Carleman's condition.}
A sequence $\bmu=(\mu_{\balpha})_{\balpha\in\N^d}$ satisfies Carleman's condition if
\begin{equation}
 \label{carleman}
\forall i=1,\ldots,d\,:\quad  \sum_{j=1}^\infty \mu(x_i^{2j})^{-1/2j}\,=\,+\infty\,.
\end{equation}
The following theorem is due to Nussbaum:
\begin{theorem}(\cite[Theorem 3.5]{CUP})
\label{th-carleman}
Let a sequence $\bmu=(\mu_{\balpha})_{\balpha\in\N^d}$ 
be such that $\M_n(\bmu)\succeq0$, for all $n\in\N$. If $\bmu$ satisfies Carleman's condition \eqref{carleman}
then $\bmu$ has a representing measure $\mu$ on $\R^d$ and $\mu$ is determinate.
\end{theorem}
A sufficient condition to ensure that a measure $\mu$ satisfies the multivariate Carleman's condition is that
\begin{equation}
\label{suff-cond}
\int \exp(c\vert x_i\vert)\,d\mu\,<\,\infty\,,\quad i=1,\ldots,d\,,\quad\mbox{for some scalar $c>0$.}
\end{equation}
\subsection{A preliminary result}

\begin{lemma}
\label{lem1}
 Let $\mu,\varphi\in\mathscr{M}(\R^d)_+$ have finite moments and assume that $\mu$ satisfies Carleman's condition \eqref{carleman}. Then 
 \begin{equation}
 \label{varphi<=mu}
 \varphi\leq\mu\quad\Leftrightarrow\quad \M_n(\varphi)\,\preceq\,\M_n(\mu)\,,\quad\forall n\in\N\,.
\end{equation}
 \end{lemma}
\begin{proof}
$\Rightarrow$ is straightforward. Indeed: 
\[\mu\geq\varphi\Rightarrow \left[\int p^2\,d\mu\,\geq\,\int p^2\,d\varphi\,,\:\forall p\in\R[\x]\,\right]\quad\Rightarrow\quad\M_n(\mu)\,\succeq\,
\M_n(\varphi)\,,\:\forall n\in\N\,.\]
$\Leftarrow$ Assume that $\M_n(\varphi)\preceq\M_n(\mu)$
for all $n\in\N$, 
and consider the sequence $\bgamma=(\gamma_{\balpha})_{\balpha\in\N^d}$, with
$\gamma_{\balpha}=\mu_{\balpha}-\varphi_{\balpha}$, for all $\balpha\in\N^d$. Then 
$\int x_i^{2n}d\varphi\leq\int x_i^{2n}d\mu$ for all $n$, and as Carleman's condition \eqref{carleman} holds for $\mu$, we infer
$\gamma(x_i^{2n})\leq \mu(x_i^{2n})$ for all $n$, and all $i=1\ldots,d$. This implies
that $\bgamma$ satisfies Carleman's condition \eqref{carleman} and therefore, as $\M_n(\bgamma)=\M_n(\mu)-\M_n(\varphi)\succeq0$ for all $n$,
by Theorem \ref{th-carleman},
$\bgamma$ has a determinate representing measure $\gamma$ on $\R^d$. In particular:
\[\int\x^{\balpha}\,d(\gamma+\varphi)\,=\,\gamma_{\balpha}+\varphi_{\balpha}\,=\,\mu_{\balpha}\,=\,\int \x^{\balpha}\,d\mu\,,\quad \forall\balpha\in\N^d\quad\Rightarrow\gamma+\varphi\,=\,\mu\,,\]
where the last statement follows from determinateness of $\mu$. Hence $\varphi\,\leq\,\mu$.
\end{proof}

\subsection{Main result}
Given two finite Borel measures $\mu$ and $\nu$ on $\R^d$, introduce the infinite-dimensional LP:

\begin{equation}
\label{LP}
\tau=\inf_{\phi^+,\phi^-\in\mathscr{M}(\R^d)_+}\,\{\,\phi^+(1)+\phi^-(1):\quad \phi_+-\phi_-\,=\,\mu-\nu\,\}\,.
\end{equation}
\begin{proposition}
 \label{prop-1}
 The LP \eqref{LP} has a unique optimal solution
 $(\phi^*_+,\phi^*_-)$ which is the Hahn-Jordan decomposition  of the signed measure $\mu-\nu$, and therefore
 $\tau=\phi^*_+(1)+\phi^*_-(1)=\Vert \mu-\nu\Vert_{TV}$.
\end{proposition}
\begin{proof}
Let $(\phi^+,\phi^-)$ be an arbitrary feasible solution of \eqref{LP}. Then as $\phi^+-\phi^-=\mu-\nu$ one obtains
$\phi^+(1)+\phi^-(1)\geq\Vert\phi^+-\phi^-\Vert_{TV}=\Vert\mu-\nu\Vert_{TV}$. On the other hand, the Hahn-Jordan decomposition $(\phi^*_+,\phi^*_-)$ of $\mu-\nu$ is feasible for \eqref{LP}, with value $\Vert\mu-\nu\Vert_{TV}$,
whence the result.
\end{proof}
Unfortunately the LP \eqref{LP} is not very useful as its stands. It is just a particular rephrasing of the total variation distance between $\mu$ and $\nu$. However we next see the a slight reinforcement of \eqref{LP}
will turn out to be very useful when passing to some hierarchy of convex relaxations. Indeed:
\begin{proposition}
\label{prop-2}
The infinite-dimensional linear program
\begin{equation}
\label{LP-new}
\rho\,=\,\inf_{\phi^+,\phi^-\in\mathscr{M}(\R^d)_+}\,\{\,\phi^+(1)+\phi^-(1):\quad \phi_+-\phi_-\,=\,\mu-\nu\,;
\quad \phi^+\leq\,\mu\,;\:\phi^-\leq\,\nu\,\}\,
\end{equation}
has same optimal value $\tau=\Vert\mu-\nu\Vert_{TV}$, and optimal solution $(\phi^*_+,\phi^*_-)$ as \eqref{LP}.
\end{proposition}
\begin{proof}
 By construction, the optimal value $\rho$ of \eqref{LP-new} satisfies $\rho\geq\tau=\Vert\mu-\nu\Vert_{TV}$. On the other hand,  with $(\phi^*_+,\phi^*_-)$ being the Hahn-Jordan decomposition of $\mu-\nu$, observe that $\phi^*_+\leq\mu$, and
 $\phi^*_-\leq\nu$. Therefore $(\phi^*_+,\phi^*_-)$ is an optimal solution of \eqref{LP-new}. Equivalently, the constraints
 $\phi^+\leq\mu$ and $\phi^-\leq\nu$ are automatically satisfied at the optimal solution $(\phi^*_+,\phi^*_-)$ of \eqref{LP} and therefore \eqref{LP} and \eqref{LP-new} have same optimal value and same optimal solution.
\end{proof}
Next, from now on we make the following assumption:
\begin{assumption}
 \label{assumption-1}
 (i) All moments of $\mu$ and $\nu$ are finite, and 
 
 (ii) $\mu$ and $\nu$ satisfy \eqref{suff-cond} (hence satisfy Carleman's condition \eqref{carleman}) for some scalar $c>0$. \end{assumption}
Consider the optimization problem
\begin{equation}
\label{LP-2}
\begin{array}{rl}
\hat{\tau}=\displaystyle\min_{\phi^+,\phi^-\in\mathscr{M}(\R^d)_+}&\{\,\phi^+(1)+\phi^-(1):\quad \phi^+-\phi^-\,=\,\mu-\nu\,;\\
&\M_n(\phi^+)\preceq\,\M_n(\mu)\,;\quad\M_n(\phi^-)\,\preceq\,\M_n(\nu)\,,\quad\forall n\in\N\,\}.
\end{array}\end{equation}
\begin{corollary}
\label{coro-TV}
 Let Assumption \ref{assumption-1} hold. 
 Then the 
 Hahn-Jordan decomposition $(\phi^*_+,\phi^*_-)$ of the signed measure $\mu-\nu$,
 is the unique optimal solution of \eqref{LP-2}, and $\hat{\tau}=\tau=\Vert\mu-\nu\Vert_{TV}$.
\end{corollary}
\begin{proof}
 By Lemma \ref{lem1}, \eqref{LP-new} and \eqref{LP-2} are equivalent.
\end{proof}

The nice feature of the LP \eqref{LP-2} when compared to its equivalent formulation \eqref{LP-new}, is that the cost as well as
the constraints of \eqref{LP-2} can next be formulated in terms of moments of $(\mu,\nu,\phi^+,\phi^-)$, so as  to yield the optimization problem:
\begin{equation}
\label{LP-3}
\begin{array}{rl}
\rho=\displaystyle\min_{\phi^+,\phi^-\in\mathscr{M}(\R^d)_+}&\{\,\phi^+(1)+\phi^-(1):\\
&\displaystyle\int \x^{\balpha}d(\phi^+-\phi^-)\,=\,
\displaystyle\int \x^{\balpha}\,d(\mu-\nu)\,,\quad\forall\balpha\in\N^d\,;\\
&\M_n(\phi^+)\,\preceq\,\M_n(\mu)\,;\:\M_n(\phi^-)\,\preceq\,\M_n(\nu)\,,\:\forall n\in\N\,\}\,,
\end{array}\end{equation}
which is an instance of the Generalized Moment Problem (GMP);
see e.g. \cite{CUP}.
\begin{corollary}
\label{corollary-2}
 Let Assumption \ref{assumption-1} hold. 
 Then the 
 Hahn-Jordan decomposition $(\phi^*_+,\phi^*_-)$ of the signed measure $\mu-\nu$,
 is the unique optimal solution of \eqref{LP-3}, and $\rho=\Vert\mu-\nu\Vert_{TV}$.
\end{corollary}
\begin{proof}
 Let $(\phi^+,\phi^-)$ be an arbitrary feasible solution of \eqref{LP-3}. By Lemma \ref{lem1}, $\phi^+\leq\mu$ and $\phi^-\leq\nu$. Hence $\phi^++\nu\leq \mu+\nu$, and 
 $\phi^-+\mu\leq \mu+\nu$. As Assumption \ref{assumption-1}(ii) holds, 
 \begin{eqnarray*}
 \int \exp(c\,\vert x_i\vert)\,d(\phi^++\nu) &<&\int \exp(c\vert x_i\vert) \,d(\mu+\nu)\,<\,\infty\\
 \int \exp(c\,\vert x_i\vert)\,d(\phi^-+\mu) &<&\int \exp(c\vert x_i\vert) \,d(\mu+\nu)\,<\,\infty\,,
 \end{eqnarray*}
 and therefore the measure $\phi^++\nu$ (resp. $\phi^-+\mu$) is determinate. But then the constraint 
 $\int x^{\balpha}d(\phi^+-\phi^-)=\int x^{\balpha}d(\mu-\nu)$ for all $\balpha\in\N^d$ reads:
 \[\int \x^{\balpha}\,d(\phi^++\nu)\,=\,
 \int \x^{\balpha}\,d(\phi^-+\mu)\,,\quad\forall\balpha\in\N^d\,,\]
 which implies $\phi^++\nu=\phi^-+\mu$ by determinacy of the measures. Therefore
 $(\phi^+,\phi^-)$ is a feasible solution of \eqref{LP-2} with same value. In other words,  \eqref{LP-3} is equivalent to \eqref{LP-2}, whence the result.
 \end{proof}

\section{A convergent hierarchy of semidefinite relaxations}
As \eqref{LP-3} is an instance of the GMP, it is natural to 
apply the Moment-SOS hierarchy \cite{HKL,ICM-2018}.
With each fixed $n\in\N$, consider the optimization problem
\begin{equation}
\label{LP-relax}
\begin{array}{rl}
\rho_n=\displaystyle\min_{\bphi,\bpsi}&\{\,\phi(1)+\psi(1):\quad
 \phi_{\balpha}-\psi_{\balpha}\,=\,\mu_{\balpha}-\nu_{\balpha}\,,\quad\forall\balpha\in\N^d_{2n}\,;\\
&0\,\preceq\,\M_n(\bphi)\preceq\,\M_n(\bmu)\,;\quad 0\,\preceq\,\M_n(\bpsi)\,\preceq\,\M_n(\bnu)\,\}\,,
\end{array}\end{equation}
where now the optimization is over degree-$2n$ pseudo-moment vectors $\bphi=(\phi_{\balpha})_{\balpha\in\N^d_{2n}}$
and $\bpsi=(\psi_{\balpha})_{\balpha\in\N^d_{2n}}$ (hence not necessarily coming from measures $\phi$ and $\psi$ on $\R^d$). Of course \eqref{LP-relax} is an obvious relaxation of \eqref{LP-3}
and therefore $\rho_n\leq \rho=\Vert\mu-\nu\Vert_{TV}$ for all $n\in\N$.

Observe that for each fixed $n\in\N$, \eqref{LP-relax} is a semidefinite program that can be solved by off-the-shelf solvers
like GloptiPoly \cite{GloptiPoly} or Jump \cite{julia} (package of the Julia programming language).

\begin{theorem}
\label{th-3}
Let Assumption \ref{assumption-1} hold.

(i) For every fixed $n\in\N$, the optimization problem \eqref{LP-relax} has an optimal solution denoted
$(\bphi^{(n)},\bpsi^{(n)})$.

(ii) In addition, 
$\rho_n\uparrow \Vert\mu-\nu\Vert_{TV}$ as $n\to\infty$, and moreover,
\begin{equation}
 \label{th3-1}
 \lim_{n\to\infty}\phi^{(n)}_{\balpha}\,=\,\int\x^{\balpha}\,d\phi^*_+\,;\quad
 \lim_{n\to\infty}\psi^{(n)}_{\balpha}\,=\,\int\x^{\balpha}\,d\phi^*_-\,,\quad\forall\balpha\in\N^d\,,
\end{equation}
where $(\phi^*_+,\phi^*_-)$ is the Hahn-Jordan decomposition of the signed measure $\mu-\nu$.
\end{theorem}
\begin{proof}
(i) Let $(\bphi,\bpsi)$ be an arbitrary  feasible solution of \eqref{LP-relax}.
As $\M_n(\bphi)\preceq\M_n(\mu)$ one obtains
  \[\phi(1)\,\leq\,\mu(1)\,;\quad \phi(x_i^{2n})\,\leq\,\mu(x_i^{2n})\,,\quad\forall i=1,\ldots,d\,,\]
 and therefore, as $\M_n(\phi^+)\succeq0$,  by \cite[Proposition 3.6]{CUP},
  \begin{equation}
 \label{rho-n}\vert\phi_{\balpha}\vert \,\leq\,\max[\mu(1)\,,\,\max_i \mu(x_i^{2n})]\,,
 \quad\forall\balpha\in\N^d_{2n}\,.
 \end{equation}
 Similarly, as $0\preceq\M_n(\bpsi)\preceq\M_n(\nu)$,
 \begin{equation}
 \label{rho'-n}
 \vert\psi_{\balpha}\vert \,\leq\,\max[\nu(1)\,,\,\max_i \nu(x_i^{2n})]\,,
 \quad\forall\balpha\in\N^d_{2n}\,.\end{equation}
 Therefore the feasible set of \eqref{LP-relax} is compact. Hence 
 \eqref{LP-relax} has an optimal solution.
 
 (ii) We proceed in two steps: We first rescale the vector
  $\bphi^{(n)}$ to $\hat{\bphi}^{(n)}$ where $\sup_{\balpha\in\N^d_{2n}}\vert\hat{\phi}^{(n)}_{\balpha}\vert\leq 1$, and complete this finite vector with zeros to make it an infinite sequence of the unit ball of the Banach space $\ell_\infty$
  of uniformly bounded infinite sequences (the dual of the Banach space $\ell_1$ 
  of infinite sequences that are summable). Then we use Banach-Alaoglu's  theorem which
  states that the unit ball of $\ell_\infty$
  is weak-star sequentially compact (i.e. compact in the weak topology $\sigma(\ell_\infty,\ell_1)$). Finally we use the fact that the limit is the same for all arbitrary converging subsequences.
 
 For each fixed $n\in\N$, and since $\M_k(\bphi^{(n)})$ is a submatrix of $\M_n(\bphi^{(n)})$ for all $k=1,\ldots,n$, again by \cite[]{CUP},
 \[  \forall\balpha\,:\:  2k-1\leq\vert\balpha\vert\leq 2k\,:\quad
 \vert\phi^{(n)}_{\balpha}\vert \,\leq\,\max[\mu(1)\,,\,\max_i \mu(x_i^{2k})]\,=:\,a_k,;\quad
  k=1,\ldots,n\,,\]
  and similarly 
   \[ \vert\psi^{(n)}_{\balpha}\vert \,\leq\,\max[\nu(1)\,,\,\max_i \nu(x_i^{2k})]\,=:\,b_k,\quad\forall\balpha\,:\:
  2k-1\leq\vert\balpha\vert\leq 2k\,;\: k=1,\ldots,n\,.\]
  
   Next, introduce the new \emph{infinite} pseudo-moment sequences:
 \begin{equation}
 \label{aux-bound-1}
 \hat{\phi}^{(n)}_{\balpha}\,:=\,\phi^{(n)}_{\balpha}/a_k\,,\quad\forall \balpha: \:2k-1\leq\vert\balpha\vert\leq 2k\,;\quad k=1,\ldots,n\,, \end{equation}
and $\hat{\phi}^{(n)}_{\balpha}=0$ for all $\balpha\in\N^d$ with $\vert\balpha\vert>2n$. Similarly,
 \begin{equation}
  \label{aux-bound-2}
 \hat{\psi}^{(n)}_{\balpha}\,:=\,\psi^{(n)}_{\balpha}/b_k\,,\quad\forall \balpha: \:2k-1\leq\vert\balpha\vert\leq 2k\,;\quad k=1,\ldots,n\,,\end{equation}
 and $\hat{\psi}^{(n)}_{\balpha}=0$ for all $\balpha\in\N^d$ with $\vert\balpha\vert>2n$. 
 
  By this re-scaling of 
  $\bphi^{(n)}=(\phi_{\balpha})_{\balpha\in\N^d_{2n}}$
 to $\hat{\bphi}^{(n)}$ (and of
  $\bpsi^{(n)}$ to $\hat{\bpsi}^{(n)}$),
both infinite sequences $\hat{\bphi}^{(n)}$ and $\hat{\bpsi}^{(n)}$ are now considered as elements of the unit ball $\B(0,1)$ of the Banach space
$\ell_\infty$ of uniformly bounded sequences, which  is sequentially compact in the $\sigma(\ell_\infty,\ell_1)$ weak topology. Therefore there exist $\hat{\bphi},\hat{\bpsi}\in\B(0,1)$ and a subsequence $(n_\ell)_{\ell\in\N}$ such that
\begin{equation}
\label{conv-1}
\lim_{\ell\to\infty}\hat{\phi}^{(n_\ell)}_{\balpha}\,=\,\hat{\phi}_{\balpha}\,;\quad
\lim_{\ell\to\infty}\hat{\psi}^{(n_\ell)}_{\balpha}\,=\,\hat{\psi}_{\balpha}\,,\quad\forall \balpha\in\N^d\,.
\end{equation}
By doing the reverse scaling of \eqref{aux-bound-1}-\eqref{aux-bound-2}, one obtains:
\begin{equation}
\label{conv-2}
\forall \balpha\in\N^d\,:\quad\lim_{\ell\to\infty}\phi^{(n_\ell)}_{\balpha}\,=\,\phi_{\balpha}\,;\quad
\lim_{\ell\to\infty}\psi^{(n_\ell)}_{\balpha}\,=\,\psi_{\balpha}\,,
\end{equation}
where for each $k\in\N$:
\[\phi_{\balpha}\,:=\,a_k\cdot\hat{\phi}_{\balpha}\,;\quad
\psi_{\balpha}\,:=\,b_k\cdot\hat{\psi}_{\balpha}\,;
\quad\forall \balpha: \:2k-1\leq\vert\balpha\vert\leq 2k\,.\]
Fix $t\in\N$ arbitrary. As $\M_t(\bphi^{(n)})\succeq0$ for all $n\geq t$, then by \eqref{conv-2},
$0\preceq\M_t(\bphi)\preceq\M_t(\mu)$, and as $t$ was arbitrary, $0\preceq\M_n(\bphi)\preceq\M_n(\mu)$ for all $n$, and
similarly $0\preceq\M_n(\bpsi)\preceq\M_n(\nu)$ for all $n$.
 Next, as $\M_n(\bphi)\preceq\M_n(\mu)$, and $\bmu$ satisfies Carleman's condition, then so does $\bphi$,
and as $\M_n(\bphi)\succeq0$ for all $n$, it follows that 
$\bphi=(\phi_{\balpha})_{\balpha\in\N^d}$ has a representing measure $\phi$ on $\R^d$. Similarly, 
$\bpsi$ has a representing measure $\psi$ on $\R^d$. 
In addition, by \eqref{conv-2},
\[\Vert\mu-\nu\Vert_{TV}\,\geq\,\lim_{\ell\to\infty}\rho_{n_\ell}\,=\,\lim_{\ell\to\infty}\phi^{(n_\ell)}(1)+\psi^{(n_\ell)}(1)\,=\,\phi(1)+\psi(1)\,,\]
\[\forall\balpha\in\N^d\,:\quad \mu_{\balpha}-\nu_{\balpha}\,=\,\lim_{\ell\to\infty}\phi^{(n_\ell)}_{\balpha}-\psi^{(n_\ell)}_{\balpha}\,=\,
\phi_{\balpha}-\psi_{\balpha}\,.\]
Hence  $(\phi,\psi)$ is an optimal solution of  \eqref{LP-2} (hence of \eqref{LP} as well), and by Corollary
\ref{coro-TV}, 
$(\phi,\psi)=(\phi^*_+,\phi^*_-)$, the Hahn-Jordan decomposition of $\mu-\nu$. 
Finally, as 
 $(n_\ell)_{\ell\in\N}$  was an arbitrary converging subsequence and the limit is independent of the subsequence,
 the whole sequence converges.
 \end{proof}

\subsection{A dual of \eqref{LP-relax}}
In this section we describe a dual of \eqref{LP-relax} and compare 
this dual to the standard dual formulation 
\begin{equation}
\label{TV-infty}
\Vert\mu-\nu\Vert_{TV}\,=\,\sup_{f\in \mathscr{B}(\R^d)}\,\{\,\int f\,d(\mu-\nu):\: \Vert f\Vert_{\infty}\,\leq\,1\,\}\,,
\end{equation}
of the TV distance. Problem  \eqref{TV-infty} is very difficult to solve because
the space $\mathscr{B}(\R^{d})$ is too large, and also because 
$\mathrm{supp}(\mu)$ and/or $\mathrm{supp}(\nu)$ are allowed to be unbounded. In fact we are not aware  of any
algorithm that can approximate the optimal value of \eqref{TV-infty} as closely as desired, at this level of generality.
On the other hand, with $n\in\N$ fixed, consider the optimization problem

\begin{equation}
\label{LP-relax-dual}
\begin{array}{rl}
\rho^*_n\,:=\displaystyle\sup_{p,\sigma_i,\psi_j} &\{\,\displaystyle\int p\,d(\mu-\nu) -\int \sigma_1\,d\mu-\int\psi_1\,d\nu:\: \\
&1-p\,=\,\sigma_0-\sigma_1\,;\: 1+p\,=\,\psi_0-\psi_1\,;\\
&p\in\R[\x]_{2n}\,;\:\sigma_i\,,\psi_i\in\Sigma[\x]_n\,,\:i=1,2\,\}\,.
\end{array}
\end{equation}
As $\sigma_i,\psi_i$ are all SOS polynomials, the constraints of \eqref{LP-relax-dual} 
\begin{equation}
\label{aux}
p\,\leq\,1+\sigma_1\quad\mbox{and}\quad -p\,\leq\,1+\psi_1\,,\quad  \forall \x\in\R^d\,,\end{equation}
imply 
\begin{equation}
\label{ersatz}
\vert\,p(\x)\,\vert\,\leq\, 1+\max[\sigma_1(\x),\psi_1(\x)]\,,\quad\forall \x\in\R^d\,.\end{equation}
and in \eqref{LP-relax-dual}, $\int \sigma_1\,d\mu+\int\psi_1\;d\nu$ is penalized in the criterion which maximizes $\int pd(\mu-\nu)$. So as the constraint $\Vert p\Vert_\infty<1$
cannot be satisfied by a polynomial $p\in\R[\x]_n$, 
one may see \eqref{aux} as a polynomial relaxation of the 
restrictive constraint  $\Vert f\Vert_\infty\leq 1$, $f\in\mathscr{B}(\R^d)$. However 
\begin{proposition}
\eqref{LP-relax-dual} is a dual of \eqref{LP-relax}, i.e., $\rho_n\geq\rho^*_n$ for every $n$.
\end{proposition}
\begin{proof}
 Let $(\bphi,\bpsi)$ and $p\in\R[\x]_{2n}$) be arbitrary feasible solutions of \eqref{LP-relax} and \eqref{LP-relax-dual} respectively. As $\int \sigma_1\,d\mu\geq\phi^+(\sigma_1)$
 and $\int \psi_1\,d\nu\geq\phi^-(\psi_1)$,
 \[\int p\,d(\mu-\nu)-\int \sigma_1\,d\mu -\int \psi_1\,d\nu\,\leq\,
 \phi^+(p)-\phi^+(\sigma_1)-\phi^-(p) -\phi^-(\psi_1)\,\]
 \[\leq\phi^+(1-\sigma_0)+\phi^-(1-\psi_0)\,\leq\,\phi^+(1)+\phi^-(1)\,,\]
 where we have used that $\phi^+(\sigma_0)\geq0$ (as 
 $\M_n(\bphi^+)\succeq0$ and $\sigma_0\in\Sigma[\x]_n$).
 This proves weak duality, i.e., $\rho_n\geq\rho^*_n$.
 \end{proof}
 We next prove that strong duality holds, i.e., there is no duality gain between \eqref{LP-relax} and its dual 
\eqref{LP-relax-dual}.  If $\mu=\nu$ then $\rho_n=\rho^*_n=0$. Indeed 
$\bphi=0=\bpsi$ is a primal feasible solution with associated $0$ value. On the other hand
$(p,\sigma_0,\sigma_1,\psi_0,\psi_1)=(1,0,0,2,0)$ is a dual feasible solution 
with associated  $0$ value.

 Otherwise, with $\mu\neq\nu$, recall that if $(\phi^+,\phi^-)$ is the Hahn-Jordan decomposition of $\mu-\nu$, then $\phi^+\leq\mu$ and $\phi^-\leq\nu$. Therefore
  \begin{equation}
  \label{HJ-decomp}
  \phi^+\,=\,f^+\,d\mu\quad\mbox{and}\quad
  \phi^-\,=\,f^-\,d\nu\,,
  \end{equation}
  for some nonnegative measurable functions $f^+,f^-$ with $f^+\leq 1$, $\mu$-a.e., and $f^-\leq 1$, $\nu$-a.e.
\begin{lemma}
 \label{duality-gap}
 Assume that there exist  open sets $O^+,O^-\subset\R^d$, with
 $O^+\subset\mathrm{supp}(\mu)$ (resp. $O^-\subset\mathrm{supp}(\nu)$).
Let $(\phi^+,\phi^-)$ be the Hahn-Jordan decomposition of $\mu-\nu$ ($\mu\neq\nu$)
and suppose that with $f^+,f^-$ as in \eqref{HJ-decomp},
$0<f^+<1$ $\mu$-a.e. (resp. $0<f^-<1$ $\nu$-a.e.) on $O_+$ (resp. $O^-$).
Then there is no duality gap between \eqref{LP-relax} and its dual \eqref{LP-relax-dual}, i.e., $\rho_n=\rho^*_n$ for all $n$ and in addition, \eqref{LP-relax-dual} has an optimal solution $(p^*,\sigma_i^*,\psi_i^*)$.
\end{lemma}
\begin{proof}
 Let $\bphi^+=(\phi^+_\balpha)_{\balpha\in\N^d_{2n}}$ and 
 $\bphi^-=(\phi^-_\balpha)_{\balpha\in\N^d_{2n}}$ be the respective moment vectors of
 $\phi^+$ and $\phi^-$ up to degree $2n$. Then $(\bphi^+,\bphi^-)$ 
 is an obvious feasible solution of \eqref{LP-relax}, and we next prove it is a strictly feasible solution. Then by our assumption  $\M_n(\phi^+)\prec\M_n(\mu)$; indeed otherwise
 suppose that $\mathrm{Ker}(\M_n(\mu)-\M_n(\phi^+))\neq\emptyset$, i.e., there exists $p\in\R[\x]_n$ such that
 \[0\,=\,\int p^2\,d(\mu-\phi^+)\,=\,\int p^2\,(1-f^+)\,d\mu\,,\]
But then one obtains the contradiction
 \[0\,=\,\int p^2\,d(\mu-\phi^+)\,\geq\,\int_{O^+} p^2\,(1-f^+)\,d\mu\,>\,0\,,\]
 as $p\neq0$ cannot vanish on an open set. For the same  reasons, $\M_n(\phi^+)\succ0$,
 and similarly $0\prec \M_n(\phi^-)\prec\M_n(\nu)$. But this strict feasibility of
 $(\bphi^+,\bphi^-)$ in \eqref{LP-relax} implies that Slater's condition holds for \eqref{LP-relax}.
 Hence by a standard results of duality for conic convex programs, 
 $\rho^*_n=\rho_n$ and since $2\geq \rho_n\geq0$,
 \eqref{LP-relax-dual} is solvable, i.e., it has an optimal solution $(p^*,\sigma_i^*,\psi_i^*)$.
 \end{proof}
In fact strong duality $\rho_n=\rho^*_n$ always hold as the set of optimal solutions
of the primal \eqref{LP-relax} is nonempty and bounded (but $\rho^*_n$ may not be attained). Indeed it follows from \eqref{rho-n}-\eqref{rho'-n}. Then invoke
\cite{strongduality} for general semidefinite programs.

\subsection{Computational remarks}
\paragraph{Moment information} 
To implement the semidefinite relaxation
\eqref{LP-relax} with fixed degree $n$, one requires knowledge 
of  the two moment sequences $(\mu_{\balpha})_{\balpha\in\N^d_{2n}}$ and
$(\nu_{\balpha})_{\balpha\in\N^d_{2n}}$, that is, all moments of $\mu$ and $\nu$ up to degree $2n$.
In some cases, all moments of $\mu$ and $\nu$ can be obtained exactly in explicit form. This is the case if $\mu$ and $\nu$ are Gaussian, or a mixture of Gaussians, or  an exponential (or a mixture of exponentials).
On the other hand, if the only information available is some sample
of i.i.d. random vectors $(X_i)_{i\leq N}$ and $(Y_i)_{i\leq N}$
drawn according to $\mu$ and $\nu$ respectively,  
then for any fixed degree $n$, we may invoke the strong law of large numbers, 
and consider the moment matrices 
$\M_n(\mu^N)$ and $\M_n(\nu^N)$ associated with the corresponding empirical measures $\mu^N$ and $\nu^N$.
By continuity of the eigenvalues, 
$\Vert\M_n(\mu)-\M_n(\mu^N)\Vert$ can be made as small a desired provided that $N$ is sufficiently large. Of course
when $n$ increases the sample size $N$ needs to be adjusted
accordingly.

If $\mu$ and $\nu$ are two probability measures, mutually singular,
then $\Vert\mu-\nu\Vert_{TV}$=2. A perfect case to check whether \eqref{LP-relax} is efficient,
is to test \eqref{LP-relax} with the toy univariate example where
$\mu=\delta_0$ and $\nu=\delta_{\varepsilon}$ for small value of $\varepsilon>0$. Indeed, one might expect
that the convergence  $\rho_n\uparrow\Vert\mu-\nu\Vert_{TV}$  as $n$ grows, could depend on
$\varepsilon$ (the smaller $\varepsilon$, the slower the convergence), or suffer from some numerical difficulties
for small $\varepsilon>0$. As seen in Example \ref{ex-0} below, this is not the case.

\subsection{Discrete (univariate) measures}
If the optimal value of \eqref{LP-relax} satisfies $\rho_n=2$ then obviously $\mu$ and $\nu$ are mutually singular.
Indeed since \eqref{LP-relax} has an optimal solution $(\bphi^*,\bpsi^*)$ with 
$\rho_n=\phi^*(1)+\psi^*(1)=2$, and since $\M_n(\bphi^*)\,{\color{red}\preceq}\,\M_n(\bmu)$ (resp. $\M_n(\bpsi^*)\preceq\M_n(\bnu)$), then necessarily
$1=\phi^*(1)=\psi^*(1)$. This implies that the vector $\bphi^*$ (resp. $\bpsi^*$) of pseudo-moments 
up to degree $2n$, is identical to $\bmu$ (i.e. that of $\mu$) (resp. $\bnu$, i.e., that of $\nu$). 
However one may ask whether such a situation happens for a finite degree $n$. We show that this is indeed the case for atomic probability measures on the real line with finite supports, in which case $n=\max[m_1,m_2]$ where
$m_1=\#\mathrm{supp}(\mu)$, and $m_2=\#\mathrm{supp}(\nu)$.

\begin{theorem}
\label{thm-singular}
Let $\mu$ and $\nu$ be  two probability measures on the real line,
supported on $X:=(x(i))_{i=1,\ldots,m_1}$
and $Y:=(y(j))_{j=1,\ldots,m_2}$ respectively.
Then with $\rho_n$ as in \eqref{LP-relax}, 
$\rho_n=\Vert\mu-\nu\Vert_{TV}$ for all $n\,\geq\,\max[m_1,m_2]$. In particular
if $X\cap Y=\emptyset$ (i.e., if $\mu$ and $\nu$ are mutually singular) then $\rho_n=2$ for all 
$n\geq \max[m_1,m_2]$.
\end{theorem}
For clarity of exposition, the proof is postponed to the appendix.

Notice that in Theorem \ref{thm-singular}, nowhere is needed an 
assumption 
on the ``distance" between points of the respective supports $X$ and $Y$ of the discrete  
measures $\mu$ and $\nu$.
However in practice, the behavior of (numerical) semidefinite software packages needed to
solve \eqref{LP-relax} is sensitive to this parameter for numerical reasons.

\begin{example}
\label{ex-0}
To illustrate Theorem \ref{thm-singular} for two mutually singular measures, consider the toy  example with  $d=1$, $\mu=\delta_0$, $\nu=\delta_{\varepsilon}$, $\varepsilon\neq0$, 
so that $\Vert\mu-\nu\Vert_{TV}=2$
 and $(\phi^*_+,\phi^*_-)=(\mu,\nu)$. The semidefinite relaxation \eqref{LP-relax} with $n=1$ reads:
 \[\begin{array}{rl}
 \rho_1=\displaystyle\min_{\bphi,\bpsi} &\{\,\phi_0+\psi_0: \phi_0=\psi_0\,;\:\phi_1-\psi_1=-\varepsilon\,;\:\phi_2-\psi_2=-\varepsilon^2\\
 &0\,\preceq\,\left [\begin{array}{cc}\phi_0 &\phi_1\\ \phi_1 & \phi_2\end{array}\right]
 \preceq \left [\begin{array}{cc}1 &0\\ 0 & 0\end{array}\right]\,;\:
  
 0\,\preceq\,\left [\begin{array}{cc}\psi_0 &\psi_1\\ \psi_1 & \psi_2\end{array}\right]
 \preceq \left [\begin{array}{cc}1 &\varepsilon\\ \varepsilon & \varepsilon^2\end{array}\right]\,\}\,.
 \end{array}\]
 The constraint $0\preceq\M_1(\bphi)\preceq\M_n(\delta_0)
 $ combined with $(0,1)\in \mathrm{Ker}(\M_1(\mu))$
 implies $(0,1)\in \mathrm{Ker}(\M_1(\bphi))$, which in turn implies $\phi_1=\phi_2=0$. Hence $\psi_1=\varepsilon$ and
 $\psi_2=\varepsilon^2$. But then $\M_1(\bpsi)\succeq0$ implies
 $\varepsilon^2\psi_0\geq\varepsilon^2$, which with $\psi_0\leq1$,  implies $\psi_0=1$ and so $\phi_0=\psi_0=1$, and $\rho_1=2$.
 
 This toy example illustrates that  in principle the first semidefinite relaxation \eqref{LP-relax} provides the optimal solution $(\phi^*_+,\phi^*_-)$, no matter how close {\color{red}$\varepsilon$  is} to $0$
 (see Theorem \ref{thm-singular}). 
 One can see here (and also in the proof of Theorem \ref{thm-singular})  how crucial for the relaxations \eqref{LP-relax} are the domination constraints   $\M_n(\bphi)\preceq\M_n(\bmu)$ and $\M_n(\bpsi)\preceq\M_n(\bnu)$, whereas they are not needed in the infinite-dimensional LP \eqref{LP}. 
 \end{example}
Theorem \ref{thm-singular} shows that
(at least in the univariate case) the semidefinite relaxations \eqref{LP-relax} obtain the exact value 
$\Vert\mu-\nu\Vert_{TV}$  as soon as $n\geq\max[m_1,m_2]$, that is, as soon as the 
minimal required moment information is used.
Moreover, nowhere in the proof 
was a condition on some minimum distance 
between  atoms of
$\mu$ and $\nu$. In fact Theorem \ref{thm-singular} and the toy illustrative example of Example 
\ref{ex-0} above, show that the atoms can be as close as desired 
without affecting the result. Of course this assertion  is only theoretical in nature and must be mitigated by the numerical behavior of the semidefinite solver in charge of solving 
the semidefinite program \eqref{LP-relax}. Indeed if some atoms are too close one should reasonably expect to encounter some numerical issues.

\subsection{Numerical examples}

In this section we provide some illustrative examples that give a first  idea of the behavior
of the moment-relaxations \eqref{LP-relax}. 

\noindent
{\bf Discrete measures.}
 To illustrate Theorem \ref{thm-singular}, we first consider the simple case of two discrete measures 
\[\mu=\frac{1}{m_1}\sum_{i=1}^{m_1}\delta_{x(i)}\,,\quad
\nu=\frac{1}{m_2}\sum_{i=1}^{m_2}\delta_{y(i)}\,.\]
In all the examples we have used the GloptiPoly 3 software for polynomial optimization \cite{GloptiPoly}
which in turns used the SeDuMi 1.03 semidefinite solver. \cite{SeDuMi}.

\begin{example}
\label{ex1}
With no point in common, i.e., $X:=\{x(i)\}$, $Y:= \{y(j)\}$ and $X\cap Y=\emptyset$, so that
$\Vert\mu-\nu\Vert_{TV}=2$ as $\mu$ and $\nu$ are mutually singular.  
Let $X=\{-1.0, 0.0, 1.0, 2.0\}$; $Y=\{-0.7, 0.3, 1.3, 2.3\}$.
Then in solving \eqref{LP-relax} with $n=4$ (i.e. with $8$ moments of $\mu$ and $\nu$), we obtain $\rho_4=1.9999$ which up to 
machine precision is considered to be $2$, as predicted by Theorem \ref{thm-singular}.
\end{example}
\begin{example}
\label{ex2}
With one point in common. If we now consider $X=\{-1.0, 0.0, 1.0, 2.0\}$ and
$Y=\{-2.0, -1.0, 0.1, 1.5\}$ so that $X\cap Y=\{-1.0\}$ and as the weights are all equal,
one obtains $\Vert\mu-\nu\Vert_{TV}=1.5$. Then with $n=4$  we obtain
$\rho_4=1.499$, which again up to machine precision can be considered as $1.5$. 
\end{example}
\begin{example}
\label{ex3}
In this example, $X=\{-1.0, 0.0, 1.0, 2.0\}$ and $Y=\{-0.7, 0.3, 1.3, 2.3\}$
(so that the points of $Y$ are ``closer" to those of $X$. 
\begin{table}[ht]
\caption{$\Vert \mu-\nu\Vert_{TV}$ for two discrete measures;
$X\cap Y=\emptyset$;
$X=\{-1.0, 0.0, 1.0, 2.0\}$; $Y=\{-0.7, 0.3, 1.3, 2.3\}$
\label{Table-1}}
\begin{tabular}{|c|c|c|}
\hline
\hline
$n$ &  4 & 5 \\
\hline
$\rho_n$ & 1.9999 & 1.9999\\
\hline
\end{tabular}
\end{table}
From results displayed in Table \ref{Table-1},
 one can see that even if some points are relatively close to each other,
the semidefinite relaxation \eqref{LP-relax} still provide a  value $\rho_n$ very close to $2$, as soon as $n\geq 4$ (i.e., with $2n=8$ moments), as predicted by Theorem \ref{thm-singular}. But now due to numerical inaccuracies
of the semidefinite solver, the resulting value is  less precise (but one can still extract
a solution $(\phi^+,\phi^-)$ very close to $(\mu,\nu)$.
\end{example}
\begin{example}
\label{ex-2}
With one point in common, i.e., $\#(X\cap Y)=1$. Let
$X=\{0.0, 0.3, 0.4, 0.9\}$ and $Y=\{0.3, 0.6, 0.7, 1.2\}$
and let the weights be  equal so that one must find 
$\Vert\mu-\nu\Vert_{TV}=1.5$.

By solving \eqref{LP-relax} with $n=4$ (i.e. with moments up to degree $8$),
one obtains $\rho_4=1.4998$ and one may extract the atoms of $\phi^*$ 
and $\psi^*$ via a subroutine of GloptiPoly.
So again, even if some points are relatively close to each other (and with $1$ point in common), the semidefinite relaxation \eqref{LP-relax} still provide a value $\rho_n=1.4998\approx 1.5$
(some unavoidable numerical inaccuracy is proper to the semidefinite solver), as soon as $n\geq 4$ (i.e., with $2n=8$ moments), as predicted by
Theorem \ref{thm-singular}.
\end{example}
\vspace{.2cm}

\noindent
{\bf Two Gaussian measures.}
We next consider the case where $\mu=\mathcal{N}(m_1,\sigma_1)$ and 
$\nu=\mathcal{N}(m_2,\sigma_2)$, and we fix
the number of moments that we consider to be $2n=4,6,8$.

\begin{table}[ht]
\caption{$\Vert \mu-\nu\Vert_{TV}$ for Gaussian measures
$\mathcal{N}(m_1,\sigma_1)$ and $\mathcal{N}(m_2,\sigma_2)$
\label{Table-3}}
\begin{tabular}{|c|c|c|c|c|c|}
\hline
\hline
$(m_1\,,\,s_1)$& $(m_2\,,\,s_2)$& $\rho_1$& $\rho_2$& $\rho_3$& $\rho_4$\\
\hline
\hline
(0\,,\,0.1)  & (1\,,\,0.1)& 1.9231& 1.9936& 1.9991& 1.9997\\
\hline
(0\,,\,0.2) & (1\,,\,0.2) & 1.7241& 1.9049& 1.9376& 1.939\\
\hline
(0\,,\,0.1) & (1\,,\, 0.5) & 1.4706& 1.6267& 1.6283& 1.7032\\
\hline
(0\,,\,0.5) & (1\,,\,0.5)& 1.0000& 1.0000& 1.1653& 1.1897\\
\hline
(0.5\,,\,0.1) & (1\,,\,0.1)& 1.7241&1.9049& 1.9375& 1.9378\\
\hline
(0.5\,,\,0.1) & (1\,,\,0.5)& 0.8197& 0.8497& 1.1249& 1.1294\\
\hline
(0.8\,,\,0.1) & (1\,,\,0.1)& 1.000& 1.0000& 1.1645& 1.1709\\
\hline
(0.8\,,\,0.05) & (1\,,\,0.1)& 1.2800&1.3507& 1.4123 &1.4290\\
\hline
(0.8\,,\,0.05) & (1\,,\,0.01)& 1.8349& 1.9616 & 1.9785& 1.9852\\
\hline
\end{tabular}
\end{table}
From results in Table \ref{Table-3} we can see the influence of a small
variance, which tends to provide $\rho_4$ with a value close to $2$, as expected
since $\mu$ and $\nu$ behave almost like the two Dirac 
measures $\delta_{m_1}$ and $\delta_{m_2}$, which are
mutually singular whenever $m_1\neq m_2$. It also turns out that 
$\rho_1$ coincides with the analytical lower bound provided in \cite{Nishiyama} on two arbitrary measures 
with given means and variances $(m_1,\sigma_1)$ and $(m_2,\sigma_2)$, namely
\[\Vert \mu-\nu\Vert_{TV}\,\geq\,2\,\frac{(m_1-m_2)^2}{(\sigma_1+\sigma_2)^2+(m_1-m_2)^2}\,.\]
(See \cite{Nishiyama}.) Notice that already with $n=2$, i.e., with moments up to degree $4$, $\rho_n$ provides a significant improvement in all cases.

For all the results, the largest size of moment matrices was $5\times 5$ and 
all the results were  obtained in less than 0.35s on a Lap-top {\tt HP Elitebook Ubuntu 24}.

\section{Conclusion} We have provided a numerical scheme to 
approximate as closely as desired the total variation distance between two measures $\mu$ and $\nu$ on $\R^d$. We have addressed this problem
under fairly general assumptions on $\mu$ and $\nu$
(Carleman's condition or the easier to check sufficient condition \eqref{suff-cond}). In particular the supports of $\mu$ and $\nu$ are \emph{not} required to be compact.
Moreover, in case where 
$\mu$ and $\nu$ are only accessible via i.i.d. samples, and for a fixed value of the degree $n$,  
the SLLN ensures that empirical moments obtained from a sufficiently large sample,
are enough for the step-$n$ semidefinite relaxation
to provide an accurate lower bound for the TV distance. Finally, even before convergence takes place,
the optimal value of each semidefinite relaxation provides a useful guaranteed lower bound
on the TV-distance, the larger $n$, the better. Of course this numerical scheme is sensitive to the dimension and so far is restricted to small dimension problems if good quality lower bounds are 
expected. (On the other hand, even crude lower bounds might be interesting in higher dimensional problems.) Therefore a topic of further investigation is to provide alternative and computationally cheaper lower bounds,
possibly at the price of loosing convergence. 

\backmatter

\section*{Declarations}

\begin{itemize}
\item Funding:
The author is supported by the AI Interdisciplinary Institute ANITI  funding through the french program
``Investing for the Future PI3A" under the grant agreement number ANR-19-PI3A-0004. This research is also part of the programme DesCartes and is supported by the National Research Foundation, Prime Minister's Office, Singapore under its Campus for Research Excellence and Technological Enterprise (CREATE) programme
\item Conflict of interest/Competing interests (check journal-specific guidelines for which heading to use): Not applicable
\item Ethics approval and consent to participate: Not applicable
\item Consent for publication: Not applicable
\item Data availability: Not applicable
\item Materials availability: Not applicable
\item Code availability: Not applicable
\item Author contribution: Not applicable
\end{itemize}

\begin{appendices}
\section{Proof of Theorem \ref{thm-singular}}

\begin{proof}
 Define the (monic) polynomials
 \begin{equation}
 \label{def-p-q}
  x\mapsto p(x)\,:=\,\prod_{i=1}^{m_1}(x-x(i))\,;\quad 
 x\mapsto q(x)\,:=\,\prod_{j=1}^{m_2}(x-y(j))\,,\end{equation}
  with respective vector of coefficients $\bp\in\R^{m_1+1}$
 and $\q\in\R^{m_2+1}$ in the usual monomial basis. 
 
 Let $(\bphi^*,\bpsi^*)$ be an optimal solution of \eqref{LP-relax} with $n=n_0:=\max[m_1,m_2]$, and w.l.o.g. 
 suppose that $n_0=m_1$. 
 Then from $\int p^2\,d\mu=0$ one deduces that $\M_{m_1}(\bmu)\bp=0$ and combining with 
 $0\preceq\M_{m_1}(\bphi^*)\preceq\M_{m_1}(\bmu)$, one also obtains 
 $\M_{m_1}(\bphi^*)\bp=0$.  Hence
 $\mathrm{rank}(\M_{m_1}(\bphi^*))=\mathrm{rank}(\M_{m_1-1}(\bphi^*))$ because 
 
 -- to every zero-eigenvector $\mathbf{h}\in\R^{m_1}$ of $\M_{m_1-1}(\bphi^*)$ (if any exists) corresponds
 a zero-eigenvector $(\mathbf{h},0)\in\R^{m_1+1}$ of $\M_{m_1}(\bphi^*)$. 
 Indeed 
 \[0\,=\,\mathbf{h}^T\M_{m_1-1}(\bphi^*)\mathbf{h}\,=\,\left(\begin{array}{c}\mathbf{h} \\ 0\end{array}\right)^T\,
 \M_{m_1}(\bphi^*)\left(\begin{array}{c}\mathbf{h} \\ 0\end{array}\right)\, \Rightarrow \M_{m_1}(\bphi^*)
 \left(\begin{array}{c}\mathbf{h} \\ 0\end{array}\right)\,=\,0\,,\]
 
 -- the vector $\bp\in\R^{m_1+1}$ of the polynomial $p\in\R[x]_{m_1}$ (and $p\not\in\R[x]_{m_1-1}$)
 is in the kernel of $\M_{m_1}(\bphi^*)$ and not in the kernel of $\M_{m_1-1}(\bphi^*)$.
 
Then by Curto and Fialkow's flat extension theorem \cite[Theorem 3.7, p. 62]{CUP}, $\bphi^*$ has 
 an atomic representing measure $\phi^*$ supported on at most $\mathrm{rank}(\M_{m_1}(\bphi^*))$ points.
 In addition $\mathrm{supp}(\phi^*)\subset X$ as $\int p^2d\phi^*=0$.
  
 Next, with $m_2\leq n=m_1$, and considering the sub-matrices $\M_{m_2-1}(\bpsi^*)$
 and $\M_{m_2}(\bpsi^*)$ as principal submatrices of $\M_n(\bpsi^*)$, a similar argument
 as above (but with $q$ instead of $p$) yields $\mathrm{rank}(\M_{m_2}(\bpsi^*))=\mathrm{rank}(\M_{m_2-1}(\bpsi^*))$.
   In addition,  if $m_2<n$ then consider the polynomials $x^kq\in\R[x]_{m_2+k}$,  with respective vectors
 $\q_k\in\R^{m_2+k+1}$, $1\leq k\leq n-m_2$.  Observe that for every
 $k$, $\M_{m_2+k}(\bpsi^*)\q_k =0$ because 
 $\M_{m_2+k}(\bpsi^*)\preceq\M_{m_2+k}(\bnu)$, and $\int q_k^2d\nu=0$).
 
 Hence $\q_k\in\mathrm{Ker}(\M_{m_2+k}(\bpsi^*))$, for every $1\leq k\leq n-m_2$,
and repeating the arguments that we have used for $\phi^*$ and $\mu$, one obtains
  $\mathrm{rank}(\M_{m_2+k}(\bpsi^*))=\mathrm{rank}(\M_{m_2}(\bpsi^*))$ for every 
 $k\leq n-m_2$.  Therefore invoking again Curto and Fialkow's flat extension theorem,
 $\bpsi^*$ has an atomic representing measure $\psi^*$ supported on at most $\mathrm{rank}(\M_{m_2}(\bpsi^*))$ points
 with $\mathrm{supp}(\psi^*)\subset Y$. Next, write
 \[\mu=\sum_{i=1}^{m_1}\alpha_i\,\delta_{x(i)}\,;\quad
 \nu=\sum_{j=1}^{m_2}\beta_j\,\delta_{y(j)}\,,\quad
 \mbox{with $\alpha_i,\beta_j>0,\:\forall i,j$}\,;\:\sum_i\alpha_i=\sum_j\beta_j=1\,,\]
 and from $\mathrm{supp}(\phi^*)\subset X$ 
 and $\mathrm{supp}(\psi^*)\subset Y$,
  we can also write
 \[\phi^*\,=\sum_{i=1}^{m_1}\alpha'_i\,\delta_{x(i)}\,;\quad
 \psi^*\,=\sum_{j=1}^{m_2}\beta'_j\,\delta_{y(j)}\,,\quad
 \mbox{with $\alpha'_i,\beta'_j\geq0,\:\forall i,j$,}\]
 and $\sum_i\alpha'_i\leq1$, $\sum_j\beta'_j\leq1$.
  Next, consider the interpolation polynomials
 \[p_i(x)\,:=\,\frac{\prod_{\ell\neq i}(x-x(\ell))}{\prod_{\ell\neq i}(x(i)-x(\ell))}\,, \quad 
  q_j(x)\,:=\,\frac{\prod_{\ell\neq j}(x-y(\ell))}
  {\prod_{\ell\neq j}(y(j)-y(\ell))}\,,\]
so that $p_i\in\R[x]_{m_1-1}$ and $q_j\in\R[x]_{m_2-1}$ 
for all $i=1,\ldots,m_1,\,j=1,\ldots,m_2$.
With $n\geq\max[m_1,m_2]$, and using $0\preceq\M_n(\bphi^*)\preceq\M_n(\bmu)$, observe that
  \[\alpha_i\,=\,\int p_i^2\,d\mu\;\geq\,\int p_i^2\,d\phi^*\quad(=\langle\bp_i,\M_n(\bphi^*)\bp_i\rangle)\,=\,\alpha'_i\,
  ,\quad \forall i=1,\ldots,m_1\,.\]
 Similarly, using $\M_n(\bpsi^*)\preceq\M_n(\bnu)$, 
  \[\beta_j\,=\,\int q_j^2\,d\nu\;\geq\,\int q_j^2\,d\psi^*\quad(=\langle\q_j,\M_n(\bpsi^*)\q_j\rangle)\,=\,\beta'_j\,
  ,\quad \forall j=1,\ldots,m_2\,.\]
  Hence we may deduce that $\phi^*\leq\mu$ and $\psi^*\leq\nu$.
  In addition, since $2m_1\leq 2n$, and as $\phi^*_j-\psi^*_j=\mu_j-\nu_j$ for all
 $j\leq 2n$,
 \[0\,=\,\int p^2\,d(\mu-\phi^*)\,=\,\int p^2\,d(\nu-\psi^*)\quad\Rightarrow\mathrm{supp}(\nu-\psi^*)\subset X\,.\]
 In particular this implies $\mathrm{supp}(\nu-\psi^*)\subset X\cap Y$
 (because $\mathrm{supp}(\nu),\mathrm{supp}(\psi^*)\subset Y$) and 
 \begin{equation}
 \label{aux-33}
 \int x^k \,p\,d(\nu-\psi^*)\,=\,0\,,\quad\forall k\in \N\,.\end{equation}
 So if $X\cap Y=\emptyset$ then necessarily $\nu-\psi^*=0$, i.e. $\psi^*=\nu$ 
 and $\psi^*(1)=1$. With similar arguments, as $2m_2\leq 2n$,
  \[0\,=\,\int q^2\,d(\nu-\psi^*)\,=\,\int q^2\,d(\mu-\phi^*)\quad\Rightarrow\mathrm{supp}(\mu-\phi^*)\subset Y\,,\]
 and in particular, $\mathrm{supp}(\mu-\phi^*)\subset X\cap Y$
 (because $\mathrm{supp}(\mu),\mathrm{supp}(\phi^*)\subset X$). Hence if $X\cap Y=\emptyset$,
 then $\phi^*=\mu$ and $\psi^*=\nu$, which yields $\phi^*(1)+\psi^*(1)=2=\Vert\mu-\nu\Vert_{TV}$.

  We want to prove that $\mu-\phi^*=\nu-\psi^*$. Indeed if true then $(\phi^*,\psi^*)$ is a feasible solution
 of \eqref{LP-new} with value $\rho_n\leq\Vert\mu-\nu\Vert_{TV}$, 
 which implies that $(\phi^*,\psi^*)$ is an optimal solution of
 \eqref{LP-new}, hence with $\rho_n=\Vert\mu-\nu\Vert_{TV}$, the desired result.

To prove that $\mu-\phi^*=\nu-\psi^*$ it is enough to prove that
\begin{equation}
\label{aux-xx}
\mu_{k}-\phi^*_{k}\,=\,\nu_{k}-\psi^*_{k}\,,\quad\forall k\in\N\,.\end{equation}
Indeed both $\mu-\phi^*$ and  $\nu-\psi^*$ are positive 
measures supported on $X$ and $Y$ respectively, hence with compact support.
Therefore if \eqref{aux-xx} holds then $\mu-\phi^*=\nu-\psi^*$
as measures on compact sets are moment determinate. 

We prove \eqref{aux-xx} by induction.
Let $j\in\N$ be fixed  and assume that $\mu_k-\phi^*_k=\nu_k-\psi^*_k$  for all $k\leq 2m_1+j\,(=2n+j)$;
by construction the statement is true for $j=0$. We next prove that
\begin{equation}
\label{aux-xxx}
\mu_{k}-\phi^*_{k}\,=\,\nu_{k}-\psi^*_{k}\,,\quad \forall  k\leq 2m_1+j+1\,.
\end{equation}
With $p$ as in \eqref{def-p-q}, write
 $p(x)=x^{m_1}-\sum_{k=0}^{m_1-1}p_k\,x^k$, so that
 \begin{equation}
 \label{aux-44}
 x^{2m_1+j+1}\,=\,x^{m_1+j+1}\,p(x)+\sum_{k=0}^{m_1-1}p_k\,x^{k+m_1+j+1}\,,\end{equation}
 and therefore, integrating with respect to the measure $\mu-\phi^*$, yields
 \begin{eqnarray*}
 \mu_{2m_1+j+1}-\phi^*_{2m_1+j+1}
 &=&\underbrace{\int x^{m_1+j+1}\,p(x)\,d(\mu-\phi^*)}_{\mbox{[$=0$ as $\mathrm{supp}(\mu),\mathrm{supp}(\phi^*)\subset X$]}}\\
&& +\sum_{k=0}^{m_1-1}p_k\,(\mu_{k+m_1+j+1}-\phi^*_{k+m_1+j+1})\\
 &=& \sum_{k=0}^{m_1-1}p_k\,(\mu_{k+m_1+j+1}-\phi^*_{k+m_1+j+1})\\
 &=& \sum_{k=0}^{m_1-1}p_k\,(\nu_{k+m_1+j+1}-\psi^*_{k+m_1+j+1})\,\:\mbox{[by induction hypothesis]}\\
 &=&  \int x^{2m_1+j+1}\,d(\nu-\psi^*)- \underbrace{\int x^{m_1+j+1} p(x)\,d(\nu-\psi^*)}_{=0\mbox{ by \eqref{aux-33}}}\quad\mbox{[ using \eqref{aux-44} ]}\\
 &=&  \int x^{2m_1+j+1}\,d(\nu-\psi^*)\,=\,\nu_{2m_1+j+1}-\psi^*_{2m_1+j+1}\,,
 \end{eqnarray*}
 which proves \eqref{aux-xxx}.  As $j\in\N$ was arbitrary, it implies \eqref{aux-xx}.
 \end{proof}
\end{appendices}

%
%
%





\bibliography{lasserrebib1}

\end{document}